\documentclass[12pt]{amsart}

\usepackage{times}
\usepackage{amssymb}
\usepackage{amsfonts}
\usepackage{amsthm}
\usepackage{graphicx}
\usepackage{psfrag}

\newcommand{\excise}[1]{}

\newtheorem{theorem}{Theorem}[section]
\newtheorem{lemma}[theorem]{Lemma}

\newtheorem{cor}[theorem]{Corollary}
\newtheorem{prop}[theorem]{Proposition}

\theoremstyle{definition}
\newtheorem{example}[theorem]{Example}
\newtheorem{remark}[theorem]{Remark}
\newtheorem{defn}[theorem]{Definition}
\newtheorem{convention}[theorem]{Convention}

\newtheorem{notation}[theorem]{Notation}

\numberwithin{equation}{section}

\newenvironment{numbered}%
        {\begin{list}
                {\noindent\makebox[0mm][r]{\arabic{enumi}.}}
                {\leftmargin=5.5ex \usecounter{enumi}}
        }
        {\end{list}}

        {\begin{list}
                {\noindent\makebox[0mm][r]{(\roman{enumi})}}
                {\leftmargin=5.5ex \usecounter{enumi}}
        }
        {\end{list}}

\newcounter{separated}

\newcommand{\ring}[1]{\ensuremath{\mathbb{#1}}}

\newcommand\<{\langle}
\renewcommand\>{\rangle}
\newcommand\CC{\ring{C}}
\newcommand\KK{{\mathcal K}}

\newcommand\NN{\ring{N}}
\newcommand\QQ{\ring{Q}}
\newcommand\RR{\ring{R}}
\newcommand\TT{{\mathcal T}}
\newcommand\ZZ{\ring{Z}}
\newcommand\kk{\Bbbk}
\newcommand\mm{{\mathfrak m}}
\newcommand\pp{{\mathfrak p}}
\newcommand\cA{{\mathcal A}}
\newcommand\cB{{\mathcal B}}
\newcommand\cC{{\mathcal C}}
\newcommand\oI{{\hspace{.25ex}\overline{\hspace{-.25ex}I}}}
\newcommand\oJ{{\hspace{.45ex}\overline{\hspace{-.45ex}J}}}

\newcommand\oU{\overline{U}}
\newcommand\wU{\widetilde{U}}
\newcommand\del{t}
\newcommand\ttt{\mathbf{t}}
\newcommand\ddel{\partial}

\renewcommand\th{{\rm th}}
\newcommand\sat{{\rm sat}}

\newcommand\rank{{\rm rank}}

\newcommand\minus{\smallsetminus}

\def\ol#1{{\overline {#1}}}

\newcommand{\aoverb}[2]{{\genfrac{}{}{0pt}{1}{#1}{#2}}}
\def\twoline#1#2{\aoverb{\scriptstyle {#1}}{\scriptstyle {#2}}}

\begin{document}

\mbox{}
\title{Combinatorics of binomial primary decomposition}

\author{Alicia Dickenstein}
\address[AD]{Departamento de Matem\'atica\\
FCEN, Universidad de Buenos Aires \\(1428) Buenos Aires, Argentina.}
\email{alidick@dm.uba.ar}
\thanks{AD was 
partially supported by UBACYT X042, CONICET PIP 5617 and ANPCyT PICT
20569, Argentina}

\author{Laura Felicia Matusevich}
\address[LFM]{Department of Mathematics \\
Texas A\&M University \\ College Station, TX 77843.}
\email{laura@math.tamu.edu}
\thanks{LFM was partially supported by an NSF Postdoctoral Research
Fellowship and NSF grant DMS-0703866}

\author{Ezra Miller}
\address[EM]{Department of Mathematics \\ University of Minnesota \\
Minneapolis, MN 55455.}
\email{ezra@math.umn.edu}
\thanks{EM was partially supported by NSF grants DMS-0304789 and DMS-0449102}

\subjclass[2000]{Primary: 13F99, 52B20, Secondary: 20M25, 14M25}
\date{25 March 2008}

\begin{abstract}
An explicit lattice point realization is provided for the primary
components of an arbitrary binomial ideal in characteristic zero.
This decomposition is derived from a char\-acter\-istic-free
combinatorial description of certain primary components of binomial
ideals in affine semigroup rings, namely those that are associated to
faces of the semigroup.  These results are intimately connected to
hypergeometric differential equations in several variables.
\end{abstract}
\maketitle

\section{Introduction}

A binomial is a polynomial with at most two terms; a binomial ideal is
an ideal generated by binomials.  Binomial ideals abound as the
defining ideals of classical algebraic varieties, particularly because
equivariantly embedded affine or projective toric varieties 
correspond to prime binomial ideals.

In fact, the zero set of any binomial ideal is a union of (translated)
toric varieties.  Thus, binomial ideals are ``easy'' in geometric
terms, and one may hope that their algebra is simple as well.  This is
indeed the case: the associated primes of a binomial ideal are
essentially toric ideals, and their corresponding primary components
can be chosen binomial as well. 
These results, due to Eisenbud and Sturmfels \cite{binomialideals},
are concrete when it comes to specifying associated primes, but less
so when it comes to 
primary components themselves, in part because of difficulty in
identifying the monomials therein.  

The main goal of this article is
to provide explicit lattice-point combinatorial realizations of the
primary components of an arbitrary binomial ideal in a polynomial ring
over an algebraically closed field of characteristic zero; this is
achieved in Theorems~\ref{t:components} and~\ref{t:toral}.  These are
proved by way of our other core result, Theorem~\ref{t:zerocomp},
which combinatorially characterizes, in the setting of an affine
semigroup ring over an arbitrary field (not required to be
algebraically closed or of characteristic zero), primary binomial
ideals whose associated prime comes from a face of the semigroup.

The hypotheses on the field~$\kk$ are forced upon us, when they occur.
Consider the univariate case: in the polynomial ring $\kk[x]$, primary
decomposition is equivalent to factorization of polynomials.
Factorization into binomials in this setting is the fundamental
theorem of algebra, which requires $\kk$ to be algebraically closed.
On the other hand, $\kk$ must have characteristic zero because of the
slightly different behavior of binomial primary decomposition in
positive characteristic \cite{binomialideals}: in characteristic zero,
every primary binomial ideal contains all of the non-monomial (i.e.,
two-term binomial) generators of its associated prime, but this is
false in positive characteristic.

The motivation and inspiration for this work came from the theory of
hypergeometric differential equations, and the results here are used
heavily in the companion article \cite{dmm} (see
Section~\ref{s:applications} for an overview of these applications).
In fact, these two projects began with a conjectural expression for
the non fully suppported solutions of a Horn hypergeometric system; its proof reduced
quickly to the statement of Corollary~\ref{c:I(B)}, which directed all
of the developments here.  Our consequent use of $M$-subgraphs
(Definition~\ref{d:M}), and more generally the application of
commutative monoid congruences toward the primary decomposition of
binomial ideals, serves as an advertisement for hypergeometric
intuition as inspiration for developments of independent interest in
combinatorics and commutative algebra.

The explicit lattice-point binomial primary decompositions in
Sections~\ref{s:primary} and~\ref{s:toral} have potential
applications beyond hypergeometric systems.  Consider the special case
of monomial ideals: certain constructions at the interface between
commutative algebra and algebraic geometry, such as integral closure
and multiplier ideals, admit concrete convex polyhedral descriptions.
The path is now open to attempt analogous constructions for binomial
ideals.

\subsection*{Overview of results}

The central combinatorial idea behind binomial primary decomposition
is elementary and concrete, so we describe it geometrically here.  The
notation below is meant to be intuitive, but in any case it coincides
with the notation formally defined later.

Fix an arbitrary binomial ideal $I$ in a polynomial ring, or more
generally in any affine semigroup ring $\kk[Q]$.  Then $I$ determines
an equivalence relation (\emph{congruence}) on~$\kk[Q]$ in which two
elements $u,v \in Q$ are congruent, written $u \sim v$, if $\ttt^u -
\lambda\ttt^v \in I$ for some $\lambda \neq 0$.  If one makes a graph
with vertex set $Q$ by drawing an edge from $u$ to $v$ when $u \sim
v$, then the congruence classes of $\sim$ are the connected components
of this graph.  For example, for an integer matrix $A$ with $n$
columns, the \emph{toric ideal}\/
\[
  I_A = \<\ttt^u - \ttt^v \mid u,v \in \NN^n \text{ and } Au = Av\>
  \subseteq \CC[\del_1,\ldots,\del_n] = \CC[\ttt] = \CC[\NN^n]
\] 
determines the congruence in which the class of $u \in \NN^n$ consists
of the set $(u + \ker(A)) \cap \NN^n$ of lattice points in the
polyhedron
$
  \{\alpha \in \RR^n \mid A\alpha = Au \text{ and } \alpha \geq 0\}.
$

The set of congruence classes for $I_A$ can be thought of as a
periodically spaced archipelago whose islands have roughly similar
shape (but steadily grow in size as the class moves interior to~$\NN^n$).
With this picture in mind, the congruence classes for any binomial
ideal~$I$ come in a finite family of such archipelagos, but instead of
each island coming with a shape predictable from its archipelago, its
boundary becomes fragmented into a number of surrounding skerries.
The extent to which a binomial ideal deviates from primality is
measured by which bridges must be built---and in which directions---to
join the islands to their skerries.

When $Q = \NN^n$ is an orthant as in the previous example, each prime
binomial ideal in $\CC[Q] = \CC[\NN^n]$ equals the sum $\pp + \mm_J$
of a prime binomial ideal~$\pp$ containing no monomials and a prime
monomial ideal~$\mm_J$ generated by the variables whose indices lie
outside of a subset $J
\hspace{-.5pt}\subseteq\hspace{-.5pt} \{1,\ldots,n\}$.  The ideal
$\pp$ is a toric ideal after rescaling the variables, so its
congruence classes are parallel to a sublattice $L = L_\pp \subseteq
\ZZ^J$; in the notation above, $L = \ker(A)$.  Now suppose that $\pp +
\mm_J$ is associated to our binomial ideal $I$.  Joining the
aforementioned skerries to their islands is accomplished by
considering congruences defined by $I + \pp$: a bridge is built from
$u$ to~$v$ whenever $u - v \in L$.

To be more accurate, just as $I + \pp$ determines a congruence
on~$\NN^n$, it determines one---after inverting the variables outside
of~$\mm_J$---on $\ZZ^J \times \NN^\oJ$, where
$\ZZ^J = \mathrm{span}_\ZZ \{e_j \mid j \in J\}$, and $\NN^{\oJ}$ is
defined analogously for the index subset~$\oJ$ complementary to~$J$.
Each resulting class in $\ZZ^J \times \NN^\oJ$ is acted upon by~$L$
and hence is a union of cosets of~$L$.  The key observation is that,
when $\pp + \mm_J$ is an associated prime of~$I$, some of these
classes consist of finitely many cosets of~$L$; let us call these
\emph{$L$-bounded} classes.  The presence of $L$-bounded classes
signals that $L$ is ``sufficiently parallel'' to the congruence
determined by~$I$, and this is how we visualize the manner in which
$\pp + \mm_J$ is associated to~$I$.

Intersecting the $L$-bounded $\ZZ^J \times \NN^\oJ$ congruence classes
with $\NN^n$ yields \emph{$L$-bounded} classes in~$\NN^n$; again,
these are constructed more or less by building bridges in directions
from~$L$ to join the classes defined by~$I$.  When the prime $\pp +
\mm_J$ is minimal over~$I$, there are only finitely many $L$-bounded
classes in $\NN^n$.  In this case, the primary component
$C_{\pp+\mm_J}$ of~$I$ is well-defined, as reflected in its
combinatorics: the congruence defined on~$\NN^n$ by $C_{\pp+\mm_J}$
has one huge class consisting of the lattice points in~$\NN^n$ lying
in no $L$-bounded class, and each of its remaining classes is
$L$-bounded in~$\NN^n$; this is the content of
Theorem~\ref{t:components}.1.  The only difference for a nonminimal
associated prime of~$I$ is that the huge class is inflated by
swal\-lowing all but a sufficiently large finite number of the
$L$-bounded classes; this is the content of the remaining parts of
Theorem~\ref{t:components}.  Here, ``sufficiently large'' means that
every swallowed $L$-bounded class contains a lattice point~$u$ with
$\ttt^u$ lying \mbox{in a fixed high power of~$\mm_J$}.

In applications, binomial ideals often arise in the presence of
multigradings.  (One reason for this is that binomial structures are
closely related to algebraic tori, whose actions on varieties induce
multigradings on coordinate rings.)  In this context, the matrix~$A$
above induces the grading: monomials $\ttt^u$ and $\ttt^v$ have equal
degree if and only if $Au = Av$.  Theorem~\ref{t:toral} expounds on
the observation that if $L = \ker(A) \cap \ZZ^J$, then a congruence
class for $I + \pp$ in $\ZZ^J \times \NN^\oJ$ is $L$-bounded if and
only if its image in $\NN^\oJ$ is finite.  This simplifies the
description of the primary components because, to describe the set of
monomials in a primary component, it suffices to refer to lattice
point geometry in $\NN^\oJ$, without mentioning $\ZZ^J \times
\NN^\oJ$.

When it comes to proofs, the crucial insight is that the geometry of
$L$-bounded classes for the congruence determined by $I + \pp$ gives
rise to simpler algebra when $\ZZ^J \times \NN^\oJ$ is reduced modulo
the action of~$L$.  Equivalently, instead of considering the
associated prime $\pp + \mm_J$ of an arbitrary binomial ideal~$I$
in~$\CC[t_1,\ldots,t_n]$, consider the prime image of the monomial
ideal $\mm_J$ associated to the image of~$I$ in $\CC[Q] =
\CC[\ttt]/\pp$, where $Q = \NN^n/L$.  Since monomial primes in an
affine semigroup ring $\CC[Q]$ correspond to faces of~$Q$, the lattice
point geometry is more obvious in this setting, and the algebra is
sufficiently uncomplicated that it works over an arbitrary field in
place of~$\CC$.  If $\Phi$ is a face of an arbitrary affine
semigroup~$Q$ whose corresponding prime $\pp_\Phi$ is minimal over a
binomial ideal~$I$ in~$\kk[Q]$, then $I$ determines a congruence on
the semigroup $Q + \ZZ\Phi$ obtained from $Q$ by allowing negatives
for~$\Phi$.  The main result in this context,
Theorem~\ref{t:zerocomp}, says that the monomials $\ttt^u$ in the
$\pp_\Phi$-primary component of~$I$ are precisely those corresponding
\enlargethispage*{0ex}
to lattice points $u \in Q$ not lying in any finite congruence class
of $Q + \ZZ\Phi$.  This, in turn, is proved by translating lattice
point geometry and combinatorics into semigroup-graded commutative
algebra in Proposition~\ref{p:graded}.

\subsection*{Acknowledgments}

This project benefited greatly from visits by its various authors to
the University of Pennsylvania, Texas A\&M University, the Institute
for Mathematics and its Applications (IMA) in Minneapolis, the
University of Minnesota, and the Centre International de Rencontres
Math\'ematiques in Luminy (CIRM).  We thank these institutions for
their gracious hospitality.

\section{Binomial ideals in affine semigroup rings}
\label{s:semigroup}

Our eventual goal is to analyze the primary components of binomial
ideals in polynomial rings over the complex numbers~$\CC$ or any
algebraically closed field of characteristic zero.  Our principal
result along these lines (Theorem~\ref{t:components}) is little more
than a rephrasing of a statement (Theorem~\ref{t:zerocomp}) about
binomial ideals in arbitrary affine semigroup rings in which the
associated prime comes from a face, combined with results of Eisenbud
and Sturmfels \cite{binomialideals}.  The developments here stem from
the observation that quotients by binomial ideals are naturally graded
by noetherian commutative monoids.  Our source for such monoids is
Gilmer's excellent book \cite{gilmer}.  For the special case of affine
semigroups, by which we mean finitely generated submonoids of free
abelian groups, see \cite[Chapter~7]{cca}.  We work in this section
over an arbitrary field~$\kk$, so it might be neither algebraically
closed nor of \mbox{characteristic~zero}.

\begin{defn}\label{d:congruence}
A \emph{congruence} on a commutative monoid~$Q$ is an equivalence
relation $\sim$ with
\[
  u \,\sim\, v\ \implies\ u\!+\!w \,\sim\, v\!+\!w \quad\text{for all
  } w \in Q.
\]
The \emph{quotient monoid} $Q/$$\sim$ is the set of equivalence
classes under addition.
\end{defn}

\begin{defn}\label{d:IQL}
The \emph{semigroup algebra} $\kk[Q]$ is the direct sum $\bigoplus_{u
\in Q} \kk \cdot \ttt^u$, with multiplication $\ttt^u \ttt^v =
\ttt^{u+v}$.  Any congruence $\sim$ on~$Q$ induces a
$(Q/$$\sim)$-grading on $\kk[Q]$ in which the \emph{monomial}\/
$\ttt^u$ has degree~$\Gamma \in Q/$$\sim$ whenever $u \in \Gamma$.  A
\emph{binomial ideal}\/ $I \subseteq \kk[Q]$ is an ideal generated by
\emph{binomials} $\ttt^u - \lambda\ttt^v$, where $\lambda \in \kk$ is
a scalar, possibly equal to zero.
\end{defn}

\begin{example}\label{ex:Q=NN}
A \emph{pure difference}\/ binomial ideal is generated by differences
of monic monomials.  Given an integer matrix $M$ with $q$~rows, we
call $I(M) \subseteq \kk[\del_1,\ldots,\del_q] = \kk[\NN^q]$ the pure
difference binomial ideal
\begin{align}\label{eq:IM}
  I(M)
  &=
  \<\ttt^u - \ttt^v \mid u - v \text{ is a column of }
  M, \, u, v \in \NN^q\> 
\\\nonumber
  &=
  \<\ttt^{w_+}-\ttt^{w_-} \mid w = w_+ - w_- \text{ is a column of }M\>.
\end{align}
Here and in the remainder of this article we adopt the convention
that, for an integer vector $w \in \ZZ^q$, the vector $w_+$ has
$i^{\rm th}$ coordinate $w_i$ if $w_i\geq 0$ and $0$ otherwise.  The
vector $w_- \in \NN^q$ is defined by $w_+ - w_- = w$, or equivalently,
$w_- = (-w)_+$.  If the columns of $M$ are linearly independent, the
ideal $I(M)$ is called a \emph{lattice basis ideal} (cf.\
Example~\ref{ex:I(B)}).  An ideal of $\kk[\del_1,\ldots,\del_q]$ has
the form described in~(\ref{eq:IM}) if and only if it is generated by
differences of monomials with disjoint~support.

The equality of the two definitions in~(\ref{eq:IM}) is easy to see:
the ideal in the first line of the display contains the ideal in the
second line by definition; and the disjointness of the supports of
$w_+$ and~$w_-$ implies that whenever $u - v = w$ is a column of~$M$,
and denoting by $\alpha:= u - w_+ = u - w_-$, we have that the
corresponding generator of the first ideal $\ttt^u - \ttt^v =
\ttt^\alpha (\ttt^{w_+} - \ttt^{w_-})$, lies in the second ideal.
\end{example}

\begin{prop}\label{p:sim}
A binomial ideal $I \subseteq \kk[Q]$ determines a congruence $\sim_I$
under which
\[
  u \sim_I v \text{\rm\ if } \ttt^u - \lambda\ttt^v \in I \text{\rm\ for
  some scalar } \lambda \neq 0.
\]
The ideal $I$ is graded by the quotient monoid $Q_I = Q/$$\sim_I$, and
$\kk[Q]/I$ has $Q_I$-graded Hil\-bert function~$1$ on every congruence
class except the class $\{u \in Q \mid \ttt^u \in I\}$ of monomials.
\end{prop}
\begin{proof}
That $\sim_I$ is an equivalence relation is because $\ttt^u -
\lambda\ttt^v \in I$ and $\ttt^v - \lambda'\ttt^w \in I$ implies
$\ttt^u - \lambda\lambda'\ttt^w \in I$.  It is a congruence because
$\ttt^u - \lambda\ttt^v \in I$ implies that $\ttt^{u+w} -
\lambda\ttt^{v+w} \in I$.  The rest is similarly straightforward.
\end{proof}

\begin{example}
In the case of a pure difference binomial ideal $I(M)$ as in
Example~\ref{ex:Q=NN}, the congruence classes under~$\sim_{I(M)}$ from
Proposition~\ref{p:sim} are the $M$-subgraphs in the following
definition, which---aside from being a good way to visualize
congruence classes---will be useful later on (see Example~\ref{ex:UM}
and Corollary~\ref{c:I(B)}, as well as Section~\ref{s:applications}).
\end{example}

\begin{defn}\label{d:M}
Any integer matrix $M$ with $q$ rows defines an undirected
graph~$\Gamma(M)$ having vertex set $\NN^q$ and an edge from $u$
to~$v$ if $u - v$ or $v - u$ is a column of~$M$.  An \emph{$M$-path}\/
from $u$ to~$v$ is a path in~$\Gamma(M)$ from $u$ to~$v$.  A subset
of~$\NN^q$ is \emph{$M$-connected}\/ if every pair of vertices therein
is joined by an $M$-path passing only through vertices in the subset.
An \emph{$M$-subgraph}\/ of~$\NN^q$ is a maximal $M$-connected subset
of~$\NN^q$ (a connected component of~$\Gamma(M)$).  An $M$-subgraph is
\emph{bounded}\/ if it has finitely many vertices, and
\emph{unbounded}\/ otherwise. (See Example~\ref{ex:concrete-subgraph}
for a concrete computation and an illustrative figure.)
\end{defn}

These $M$-subgraphs bear a marked resemblance to the concept of
\emph{fiber}\/ in \cite[Chapter~4]{gb&cp}.  The interested reader will
note, however, that even if these two notions have the same flavor,
their definitions have mutually exclusive assumptions, since for a
square matrix~$M$, the corresponding matrix~$A$ in \cite{gb&cp} is
empty.

Given a face~$\Phi$ of an affine semigroup $Q \subseteq \ZZ^\ell$, the
\emph{localization} of~$Q$ along~$\Phi$ is the affine semigroup $Q +
\ZZ\Phi$ obtained from~$Q$ by adjoining negatives of the elements
in~$\Phi$.  The algebraic version of this notion is a common tool for
affine semigroup rings \cite[Chapter~7]{cca}: for each $\kk[Q]$-module
$V$, let $V[\ZZ\Phi]$ denote its \emph{homogeneous localization}\/
along~$\Phi$, obtained by inverting $\ttt^\phi$ for all $\phi \in
\Phi$.  For example, $\kk[Q][\ZZ\Phi] \cong \kk[Q+\ZZ\Phi]$.  Writing
\[
  \pp_\Phi 
   = {\rm span}_{\kk} \{ \ttt^u \mid u \in Q \minus \Phi \} 
   \subseteq \kk[Q]
\]
for the prime ideal of the face~$\Phi$, so that $\kk[Q]/\pp_\Phi =
\kk[\Phi]$ is the affine semigroup ring for~$\Phi$, we find, as a
consequence, that $\pp_\Phi[\ZZ\Phi] = \pp_{\ZZ\Phi} \subseteq
\kk[Q+\ZZ\Phi]$, because
\[
\kk[Q+\ZZ\Phi]/\pp_\Phi[\ZZ\Phi] = (\kk[Q]/\pp_\Phi)[\ZZ\Phi]= 
 \kk[\Phi][\ZZ\Phi] = \kk[\ZZ\Phi].
\]
(We write equality signs to denote canonical isomorphisms.)  For any
ideal $I \subseteq \kk[Q]$, the localization $I[\ZZ\Phi]$ equals the
extension $I\, \kk[Q+\ZZ\Phi]$ of~$I$ to $\kk[Q+\ZZ\Phi]$, and we write
\begin{equation}\label{eq:I:Phi}
  (I:\ttt^\Phi) = I[\ZZ\Phi] \cap \kk[Q],
\end{equation}
the intersection taking place in $\kk[Q+\ZZ\Phi]$.  Equivalently,
$(I:\ttt^\Phi)$ is the usual colon ideal $(I:\ttt^\phi)$ for any
element $\phi$ sufficiently interior to~$\Phi$ (for example, take
$\phi$ to be a high multiple of the sum of the generators of~$\Phi$);
in particular, $(I:\ttt^\Phi)$ is a binomial ideal when $I$ is.

For the purpose of investigating $\pp_\Phi$-primary components, the
ideal $(I:\ttt^\Phi)$ is as good as~$I$ itself, since this colon
operation does not affect such components, or better, since the
natural map from $\kk[Q]/(I:\ttt^\Phi)$ to its homogeneous
localization along~$\Phi$ is injective.  Combinatorially, what this
means is the following.

\begin{lemma}
A subset $\Gamma' \subseteq Q$ is a congruence class in
$Q_{(I:\ttt^\Phi)}$ determined by $(I:\ttt^\Phi)$ if and only if\/
$\Gamma' = \Gamma \cap Q$ for some class $\Gamma \subseteq Q+\ZZ\Phi$
under the congruence~$\sim_{I[\ZZ\Phi]}$.\qed
\end{lemma}

\begin{lemma}\label{l:distinct}
If a congruence class $\Gamma \subseteq Q+\ZZ\Phi$ under
$\sim_{I[\ZZ\Phi]}$ has two distinct elements whose difference lies
in~$Q+\ZZ\Phi$, then for all $u \in \Gamma$ the monomial $\ttt^u$ maps
to~$0$ in the (usual inhomogeneous) localization
$(\kk[Q]/I)_{\pp_\Phi}$ inverting all elements not in~$\pp_\Phi$.
\end{lemma}
\begin{proof}
Suppose $v \not=w \in \Gamma$ with $w-v \in Q+\ZZ\Phi$.  The images
in~$\kk[Q]/I$ of the monomials $\ttt^u$ for $u \in \Gamma$ are nonzero
scalar multiples of each other, so it is enough to show that $\ttt^v$
maps to zero in~$(\kk[Q]/I)_{\pp_\Phi}$.  Since $w-v \in Q + \ZZ\Phi$,
we have $\ttt^{w-v} \in \kk[Q+\ZZ\Phi]$.  Therefore
$1-\lambda\ttt^{w-v}$ lies outside of~$\pp_{\ZZ\Phi}$ for all $\lambda
\in \kk$, because its image in~$\kk[\ZZ\Phi] = \kk[Q+ \ZZ
\Phi]/\pp_{\ZZ\Phi}$ is either $1-\lambda\ttt^{w-v}$ or~$1$, according
to whether or not $w-v \in \ZZ\Phi$.  (The assumption $v \neq w$ was
used here: if $v=w$, then for $\lambda=1$, we have $1-\lambda
\ttt^{w-v} = 0 $.)  Hence $1-\lambda\ttt^{w-v}$ maps to a unit
in~$(\kk[Q]/I)_{\pp_\Phi}$.  It follows that $\ttt^v$ maps to~$0$,
since $(1-\lambda_{vw}\ttt^{w-v})\ttt^v = \ttt^v-\lambda_{vw}\ttt^w$
maps to~$0$ in~$\kk[Q]/I$ whenever $\ttt^v - \lambda_{vw} \ttt^w \in
I$.
\end{proof}

\begin{lemma}\label{l:unbounded}
A congruence class $\Gamma \subseteq Q+\ZZ\Phi$ under
$\sim_{I[\ZZ\Phi]}$ is infinite if and only if it contains two
distinct elements whose difference lies in\/~$Q+\ZZ\Phi$.
\end{lemma}
\begin{proof}
Let $\Gamma \subseteq Q+\ZZ\Phi$ be a congruence class.  If $v,w \in
\Gamma$ and $v-w \in Q+\ZZ\Phi$, then $w + \epsilon (v-w) \in \Gamma$
for all positive $\epsilon \in \ZZ$.  On the other hand, assume
$\Gamma$ is infinite.  There are two possibilities: either there are
$v,w \in \Gamma$ with $v-w \in \ZZ\Phi$, or not.  If so, then we are
done, so assume not.  Let $\ZZ^q$ be the quotient of
$\ZZ^\ell/\ZZ\Phi$ modulo its torsion subgroup. (Here $\ZZ^\ell$ is
the ambient lattice of $Q$.)  The projection $\ZZ^\ell \to \ZZ^q$
induces a map from $\Gamma$ to its image~$\ol\Gamma$ that is
finite-to-one.  More precisely, if $\Gamma'$ is the intersection
of~$\Gamma$ with a coset of $\ZZ\Phi$ in~$\ker(\ZZ^\ell \to \ZZ^q)$,
then $\Gamma'$ maps bijectively to its image~$\ol\Gamma{}'$.  There
are only finitely many cosets, so some~$\Gamma'$ must be infinite,
along with~$\ol\Gamma{}'$.  But $\ol\Gamma{}'$ is a subset of the
affine semigroup~$\ol{Q/\Phi}$, defined as the image of $Q + \ZZ\Phi$
in~$\ZZ^q$.  As $\ol{Q/\Phi}$ has unit group zero, every infinite
subset contains two points whose difference lies in~$\ol{Q/\Phi}$, and
the corresponding lifts of these to~$\Gamma'$ have their difference in
$Q + \ZZ\Phi$.
\end{proof}

\begin{defn}\label{d:U}
Fix a face $\Phi$ of an affine semigroup~$Q$.  A subset $S \subseteq
Q$ is an \emph{ideal}\/ if $Q+S \subseteq S$, and in that case we
write $\kk\{S\} = \<\ttt^u \mid u \in S\> = {\rm span}_\kk\{\ttt^u
\mid u \in S\}$ for the monomial ideal in~$\kk[Q]$ having $S$ as its
$\kk$-basis.  An ideal~$S$ is \emph{$\ZZ\Phi$-closed}\/ if $S = Q \cap
(S + \ZZ\Phi)$.  If~$\sim$ is a congruence on $Q+\ZZ\Phi$, then the
\emph{unbounded ideal}\/ $U \subseteq Q$ is the ($\ZZ\Phi$-closed)
ideal of elements $u \in Q$ with infinite congruence class
under~$\sim$ in $Q + \ZZ\Phi$.  Finally, write $\cB(Q+\ZZ\Phi)$ for
the set of bounded (i.e.\ finite) congruence classes of $Q+\ZZ\Phi$
under $\sim$.
\end{defn}

\begin{example}\label{ex:UM}
Let $M$ be as in Definition~\ref{d:M} and consider the congruence
$\sim_{I(M)}$ on $Q = \NN^q$.  If $\Phi = \{0\}$, then the unbounded
ideal $U \subseteq \NN^q$ is the union of the unbounded $M$-subgraphs
of~$\NN^q$, while $\cB(\NN^q)$ is the union of the bounded
$M$-subgraphs.
\end{example}

\begin{prop}\label{p:graded}
Fix a face $\Phi$ of an affine semigroup~$Q$, a binomial ideal $I
\subseteq \kk[Q]$, and a $\ZZ\Phi$-closed ideal $S \subseteq Q$
containing $U$ under the congruence $\sim_{I[\ZZ\Phi]}$.  Write $\cB =
\cB(Q+\ZZ\Phi)$ for the bounded classes, $J$ for the binomial ideal
$(I:\ttt^\Phi) + \kk\{S\}$, and $\ol Q = (Q+\ZZ\Phi)_{I[\ZZ\Phi]}$.
\begin{numbered}
\item
$\kk[Q]/J$ is graded by $\ol Q$, and its set of nonzero degrees is
contained in~$\cB$.
\item
The group $\ZZ\Phi \subseteq \ol Q$ acts freely on~$\cB$, and the
$\kk[\Phi]$-submodule $(\kk[Q]/J)_T \subseteq \kk[Q]/J$ in degrees
from any orbit $T \subseteq \cB$ is~$0$ or finitely generated and
torsion-free of rank~$1$.
\item
The quotient $Q_J/\Phi$ of the monoid $(Q+\ZZ\Phi)_{J[\ZZ\Phi]}$ by
its subgroup $\ZZ\Phi$ is a partially ordered set if we define $\zeta
\preceq \eta$ whenever $\zeta + \xi = \eta\,$ for some $\xi \in
Q_J/\Phi$.
\item
$\kk[Q]/J$ is filtered by $\ol Q$-graded\/ $\kk[Q]$-submodules with
associated graded module
\[
  {\rm gr}(\kk[Q]/J) = \bigoplus_{T \in \cB/\Phi} (\kk[Q]/J)_T, \quad
  \text{where } \cB/\Phi = \{\ZZ\Phi\text{\rm -orbits }T \subseteq \cB\},
\]
the canonical isomorphism being as $\cB$-graded $\kk[\Phi]$-modules,
although the left-hand side is naturally a $\kk[Q]$-module annihilated
by~$\pp_\Phi$.
\item
If\/ $(\kk[Q]/J)_T \neq 0$ for only finitely many orbits $T \in
\cB/\Phi$, then $J$ is a $\pp_\Phi$-primary ideal.
\end{numbered}
\end{prop}
\begin{proof}
The quotient $\kk[Q]/(I:\ttt^\Phi)$ is automatically $\ol Q$-graded by
Proposition~\ref{p:sim} applied to $Q+\ZZ\Phi$ and $I[\ZZ\Phi]$,
given~(\ref{eq:I:Phi}).  The further quotient by~$\kk\{S\}$ is graded
by $\cB$ because~\mbox{$S \supseteq U$}.

$\ZZ\Phi$ acts freely on~$\cB$ by Lemmas~\ref{l:distinct}
and~\ref{l:unbounded}: if $\phi \in \ZZ\Phi$ and $\Gamma$ is a bounded
congruence class, then the translate $\phi + \Gamma$ is, as well; and
if $\phi \neq 0$ then $\phi + \Gamma \neq \Gamma$, because each coset
of~$\ZZ\Phi$ intersects $\Gamma$ at most once.  Combined with the
$\ZZ\Phi$-closedness of~$S$, this shows that $\kk[Q]/J$ is a
$\kk[\Phi]$-submodule of the free $\kk[\ZZ\Phi]$-module whose basis
consists of the $\ZZ\Phi$-orbits $T \subseteq \cB$.  Hence
$(\kk[Q]/J)_T$ is torsion-free (it might be zero, of course, if $S$
happens to contain all of the monomials corresponding to congruence
classes of~$Q$ arising from $\sim_{I[\ZZ\Phi]}$ classes in~$T$).  For
item~2, it remains to show that $(\kk[Q]/J)_T$ is finitely generated.
Let \mbox{$\TT = \bigcup_{\Gamma \in T} \Gamma \cap Q$}.  By
construction, $\TT$ is the (finite) union of the intersections $Q \cap
(\gamma + \ZZ\Phi)$ of $Q$ with cosets of~$\ZZ\Phi$ in~$\ZZ^\ell$ for
$\gamma$ in any fixed $\Gamma \in T$.  Such an intersection is a
finitely generated $\Phi$-set (a set closed under addition by~$\Phi$)
by \cite[Eq.~(1) and Lemma~2.2]{irredres} or
\cite[Theorem~11.13]{cca}, where the $\kk$-vector space it spans is
identified as the set of monomials annihilated by~$\kk[\Phi]$ modulo
an irreducible monomial ideal of~$\kk[Q]$.  The images in $\kk[Q]/J$
of the monomials corresponding to any generators for these $\Phi$-sets
generate $(\kk[Q]/J)_T$.

The point of item~3 is that the monoid $Q_J/\Phi$ acts sufficiently
like an affine semigroup whose only unit is the trivial one.  To prove
it, observe that $Q_J/\Phi$ consists, by item~1, of the (possibly
empty set of) orbits $T \in \cB$ such that $(\kk[Q]/J)_T \neq 0$ plus
one congruence class $\ol S$ for the monomials in~$J$ (if there are
any).  The proposed partial order has $T \prec \ol S$ for all orbits
$T \in Q_J/\Phi$, and also $T \prec T+v$ if and only if $v \in
(Q+\ZZ\Phi) \minus \ZZ\Phi$.  This relation~$\prec$ a~priori defines a
directed graph with vertex set~$Q_J/\Phi$, and we need it to have no
directed cycles.  The terminal nature of~$\ol S$ implies that no cycle
can contain~$\ol S$, so suppose that $T = T+v$.  For some $\phi \in
\ZZ\Phi$, the translate $u + \phi$ lies in the same congruence class
under $\sim_{I[\ZZ\Phi]}$ as $u + v$.  Lemma~\ref{l:unbounded} implies
that $v-\phi$, and hence $v$ itself, does not lie in $Q+\ZZ\Phi$.

For item~4, it suffices to find a total order $T_0,T_1,T_2,\ldots$ on
$\cB/\Phi$ such that \mbox{$\bigoplus_{j \geq k} (\kk[Q]/J)_{T_j}$} is a
$\kk[Q]$-submodule for all $k \in \NN$.  Use the partial order of
$\cB/\Phi$ via its inclusion in the monoid $Q_J/\Phi$ in item~3 for $S =
U$.  Any well-order refining this partial order will do.

Item~5 follows from items~2 and~4 because the associated primes of
${\rm gr}(\kk[Q]/J)$ contain every associated prime of~$J$ for any
finite filtration of $\kk[Q]/J$ by $\kk[Q]$-submodules.
\end{proof}

For connections with toral modules (Definition~\ref{d:toral}), we
record the following.

\begin{cor}\label{c:graded}
Fix notation as in Proposition~\ref{p:graded}.  If $I$ is homogeneous
for a grading of\/~$\kk[Q]$ by a group~$\cA$
via a monoid morphism $Q \to \cA$, then $\kk[Q]/J$ and ${\rm
gr}(\kk[Q]/J)$ are $\cA$-graded via a natural coarsening $\cB \to \cA$
that restricts to a group homomorphism $\ZZ\Phi \to \cA$.
\end{cor}
\begin{proof}
The morphism $Q \to \cA$ induces a morphism $\pi_\cA: Q+\ZZ\Phi \to
\cA$ by the universal property of monoid localization.  The
morphism~$\pi_\cA$ is constant on the non-monomial congruence classes
in~$Q_I$ precisely because $I$ is $\cA$-graded.  It follows that
$\pi_\cA$ is constant on the non-monomial congruence classes
in~$(Q+\ZZ\Phi)_{I[\ZZ\Phi]}$.  In particular, $\pi_\cA$ is constant
on the bounded classes $\cB(Q+\ZZ\Phi)$, which therefore map to~$\cA$ to
yield the natural coarsening.  The group homomorphism $\ZZ\Phi \to
\cA$ is induced by the composite morphism $\ZZ\Phi \to (Q+\ZZ\Phi) \to
\cA$, which identifies the group $\ZZ\Phi$ with the $\ZZ\Phi$-orbit
in~$\cB$ containing (the class of)~$0$.
\end{proof}

\begin{theorem}\label{t:zerocomp}
Fix a face $\Phi$ of an affine semigroup~$Q$ and a binomial ideal $I
\subseteq \kk[Q]$.  If\/ $\pp_\Phi$ is minimal over~$I$, then the
$\pp_\Phi$-primary component of~$I$ is $(I:\ttt^\Phi) + \kk\{U\}$,
where $(I:\ttt^\Phi)$ is the binomial ideal~(\ref{eq:I:Phi}) and $U
\subseteq Q$ is the unbounded ideal (Definition~\ref{d:U}) for
$\sim_{I[\ZZ\Phi]}$.  Furthermore, the only monomials in
$(I:\ttt^\Phi) + \kk\{U\}$ are those of the form $\ttt^u$ for $u \in
U$.
\end{theorem}
\begin{proof}
The $\pp_\Phi$-primary component of~$I$ is the kernel of the
localization homomorphism $\kk[Q] \to (\kk[Q]/I)_{\pp_\Phi}$.  As this
factors through the homogeneous localization
$\kk[Q+\ZZ\Phi]/I[\ZZ\Phi]$, we find that the kernel contains
$(I:\ttt^\Phi)$.  Lemmas~\ref{l:distinct} and~\ref{l:unbounded} imply
that the kernel contains~$\kk\{U\}$.  But already $(I:\ttt^\Phi) +
\kk\{U\}$ is $\pp_\Phi$-primary by Proposition~\ref{p:graded}.5; the
finiteness condition there is satisfied by minimality of~$\pp_\Phi$
applied to the filtration in Proposition~\ref{p:graded}.4.
Thus the quotient of $\kk[Q]$ by $(I:\ttt^\Phi) + \kk\{U\}$ maps
injectively to its localization at~$\pp_\Phi$.  To prove the last
sentence of the theorem, observe that under the $\ol Q$-grading from
Proposition~\ref{p:graded}.1, every monomial~$\ttt^u$ outside
of~$\kk\{U\}$ maps to a $\kk$-vector space basis for the
($1$-dimensional) graded piece corresponding to the bounded congruence
class containing~$u$.
\end{proof}

\begin{example}
One might hope that when $\pp_\Phi$ is an embedded prime of a binomial
ideal~$I$, the $\pp_\Phi$-primary components, or even perhaps the
irreducible components, would be unique, if we require that they be
finely graded (Hilbert function $0$ or~$1$) as in
Proposition~\ref{p:graded}.  However, this fails even in simple
examples, such as $\kk[x,y]/\<x^2-xy,xy-y^2\>$.  In this case, $I =
\<x^2-xy,xy-y^2\> = \<x^2,y\> \cap \<x-y\> = \<x,y^2\> \cap \<x-y\>$
and $\Phi$ is the face $\{0\}$ of~$Q = \NN^2$, so that $I =
(I:\ttt^\Phi)$ by definition.  The monoid $Q_I$, written
multiplicatively, consists of $1$, $x$, $y$, and a single element of
degree~$i$ for each $i \geq 2$ representing the congruence class of
the monomials of total degree~$i$.  Our two choices $\<x^2,y\>$ and
$\<x,y^2\>$ for the irreducible component with associated
prime~$\<x,y\>$ yield quotients of~$\kk[x,y]$ with different
$Q_I$-graded Hilbert functions, the first nonzero in degree~$x$ and
the second nonzero in degree~$y$.
\end{example}

\section{Primary components of binomial ideals}
\label{s:primary}

In this section, we express the primary components of binomial ideals
in polynomial rings over the complex numbers as explicit sums of
binomial and monomial ideals.  We formulate our main result,
Theorem~\ref{t:components}, after recalling some essential results
from \cite{binomialideals}.  In this section we work with the complex
polynomial ring $\CC[\ttt]$ in (commuting)
variables $\ttt = \del_1,\ldots,\del_n$.

If $L \subseteq \ZZ^n$ is a sublattice, then with notation as in
Example~\ref{ex:Q=NN}, the \emph{lattice ideal}\/ of~$L$ is
\[
  I_L = \<\ttt^{u_+} - \ttt^{u_-} \mid u = u_+ - u_- \in L\>,
\]
More generally, any \emph{partial character} $\rho : L \to \CC^*$
of~$\ZZ^n$, which includes the data of both its domain lattice $L
\subseteq \ZZ^n$ and the map to~$\CC^*$, determines a binomial ideal
\[
  I_\rho = \<\ttt^{u_+} - \rho(u)\ttt^{u_-} \mid u = u_+ - u_- \in L\>.
\]
(The ideal~$I_\rho$ is called $I_+(\rho)$ in \cite{binomialideals}.)
The ideal $I_\rho$ is prime if and only if $L$ is a \emph{saturated}
sublattice of~$\ZZ^n$, meaning that $L$ equals its
\emph{saturation}, in general defined as
\[
  \sat(L) = (\QQ L) \cap \ZZ^n,
\]
where $\QQ L = \QQ \otimes_\ZZ L$ is the rational vector space spanned
by~$L$ in $\QQ^n$.  In fact, writing $\mm_J = \<\del_j \mid j \notin
J\>$ for any $J \subseteq \{1,\ldots,n\}$, every binomial prime ideal
in $\CC[\ttt]$ has the form
\begin{equation} 
\label{eq:IrJ}
  I_{\rho,J} = I_\rho + \mm_J
\end{equation}
for some \emph{saturated}\/ partial character~$\rho$ (i.e., whose
domain is a saturated sublattice) and subset $J$ such that the
binomial generators of~$I_\rho$ only involve variables $\del_j$ for $j
\in J$ (some of which might actually be absent from the generators
of~$I_\rho$) \cite[Corollary~2.6]{binomialideals}.

\begin{remark}\label{rem:I_A}
A rank $m$ lattice $L \subseteq \ZZ^n$ is saturated if and only if
there exists an $(n-m)\times n$ integer matrix $A$ of full rank such
that $L=\ker_{\ZZ}(A)$.  In this case, if $\rho$ is the trivial
character, the ideal $I_{\rho}$ is denoted by $I_A$ and called a
\emph{toric ideal}.  Note that
\begin{equation}
\label{eq:IA}
I_A = \langle \ttt^u - \ttt^v \mid Au = Av \rangle.
\end{equation}
If $\rho$ is not the trivial character, then $I_\rho$ becomes  
isomorphic to~$I_A$ when the variables are rescaled via $\del_i
\mapsto \rho(e_i)\del_i$, which induces the rescaling $\ttt^u
\mapsto \rho(u)\ttt^u$ on general monomials.
\end{remark}

The characteristic zero part of the main result in
\cite{binomialideals}, Theorem~7.1$'$, says that an irredundant
primary decomposition of an arbitrary binomial ideal $I \subseteq
\CC[\ttt]$ is given by
\begin{equation}\label{eq:Hull}
  I = \bigcap_{I_{\rho,J}\in\mathrm{Ass}(I)}\mathrm{Hull}(I + I_\rho +
  \mm_J^e)
\end{equation}
for any large integer~$e$, where $\mathrm{Hull}$ means to discard the
primary components for embedded (i.e.\ nonminimal associated) primes,
and $\mm_J^e = \<\del_j \mid j \notin J\>^e$.  Our goal in this
section is to be explicit about the Hull operation.  The salient
feature of~(\ref{eq:Hull}) is that $I + I_\rho + \mm_J^e$
contains~$I_\rho$.  In contrast, (\ref{eq:Hull}) is false in positive
characteristic, where $I_\rho + \mm_J^e$ should be replaced by a
Frobenius power of~$I_{\rho,J}$ \cite[Theorem~7.1$'$]{binomialideals}.

Our notation in the next theorem is as follows.  Given a subset $J
\subseteq \{1,\ldots,n\}$, let $\oJ = \{1,\ldots,n\} \minus J$ be its
complement, and use these sets to index coordinate subspaces
of~$\NN^n$ and $\ZZ^n$; in particular, $\NN^n = \NN^J \times \NN^\oJ$.
Adjoining additive inverses for the elements in $\NN^J$ yields $\ZZ^J
\times \NN^\oJ$, whose semigroup ring we denote by
$\CC[\ttt][\ttt_J^{-1}]$, with $\ttt_J = \prod_{j \in J} \del_j$.  As in
Definition~\ref{d:U}, $\CC\{S\}$ is the monomial ideal in~$\CC[\ttt]$
having $\CC$-basis~$S$.  Finally, for a saturated sublattice $L
\subseteq \ZZ^J$, we write $\NN^J/L$ for the image of $\NN^J$ in the
torsion-free group~$\ZZ^J/L$.

\begin{theorem}\label{t:components}
Fix a binomial ideal $I \subseteq \CC[\ttt]$ and an associated prime
$I_{\rho,J}$ of~$I$, where $\rho : L \to \CC^*$ for a saturated
sublattice $L \subseteq \ZZ^J \subseteq \ZZ^n$.  Set $\Phi = \NN^J/L$,
and write $\sim$ for the congruence on $\ZZ^J \times \NN^\oJ$
determined by the ideal $(I+I_\rho)[\ZZ^J] =
(I+I_\rho)\CC[\ttt][\ttt_J^{-1}]$.
\begin{numbered}
\item
If $I_{\rho,J}$ is a minimal prime of~$I$ and $\wU$ is the set of $u
\in \NN^n$ whose congruence classes in $(\ZZ^J \times \NN^\oJ)/$$\sim$
have infinite image in~$\ZZ\Phi \times \NN^\oJ$, then the
$I_{\rho,J}$-primary component of~$I$~is
\[
  \cC_{\rho,J} = \big( (I + I_\rho) : \ttt_J^\infty\big) +
  \CC\big\{\wU\big\}.
\]
\end{numbered}\setcounter{separated}{\value{enumi}}%
Fix a monomial ideal $K \subseteq \CC[\del_j \mid j \in \oJ]$ containing
a power of each available variable, and let $\approx$ be the
congruence on $\ZZ^J \times \NN^\oJ$ determined by
$(I+I_\rho+K)\CC[\ttt][\ttt_J^{-1}]$.  Write $\wU_K$ for the set of $u
\in \NN^n$ whose congruence classes in $(\ZZ^J \times
\NN^\oJ)/$$\approx$ have infinite image in~$\ZZ\Phi \times \NN^\oJ$.
\begin{numbered}\setcounter{enumi}{\value{separated}}
\item
The $I_{\rho,J}$-primary component of $\<I + I_\rho + K\> \subseteq
\CC[\ttt]$ is $\big( (I + I_\rho + K) : \ttt_J^\infty \big)+\CC\{\wU_K\}$.
\item
If $K$ is contained in a sufficiently high power of\/ $\mm_J$, then
\[
  \cC_{\rho,J} = \big( (I + I_\rho + K) : \ttt_J^\infty\big) +
  \CC\big\{\wU_K\big\}
\]
is a valid choice of $I_{\rho,J}$-primary component for~$I$.
\end{numbered}
The only monomials in the above primary components are those in
$\CC\{\wU\}$ or $\CC\{\wU_K\}$.
\end{theorem}
\begin{proof}
First suppose $I_{\rho,J}$ is a minimal prime of~$I$.  We may, by
rescaling the variables $\del_j$ for $j \in J$, harmlessly assume that
$\rho$ is the trivial character on its lattice~$L$, so that $I_\rho =
I_L$ is the lattice ideal for~$L$.  The quotient $\CC[\ttt]/I_L$ is
the affine semigroup ring~$\CC[Q]$ for $Q = (\NN^J/L) \times \NN^\oJ =
\Phi \times \NN^\oJ$.  Now let us take the whole situation
modulo~$I_L$.  The image of~$I_{\rho,J} = I_L + \mm_J$ is the prime
ideal $\pp_\Phi \subseteq \kk[Q]$ for the face~$\Phi$.  The image in
$\CC[Q]$ of the binomial ideal~$I$ is a binomial ideal~$I'$, and
$\big( (I + I_L) : \ttt_J^\infty \big)$ has image $(I' : \ttt^\Phi)$,
as defined in~(\ref{eq:I:Phi}).  Finally, the image of $\wU$ in~$Q$ is
the unbounded ideal $U \subseteq Q$ (Definition~\ref{d:U}) by
construction.  Now we can apply Theorem~\ref{t:zerocomp} to~$I'$ and
obtain a combinatorial description of the component associated to
$I_{\rho,J}$.

The second and third items follow from the first by replacing $I$ with
$I+K$, given the primary decomposition in~(\ref{eq:Hull}).
\end{proof}

\begin{remark}
One of the mysteries in \cite{binomialideals} is why the primary
components~$\cC$ of binomial ideals turn out to be generated by
monomials and binomials.  {}From the perspective of
Theorem~\ref{t:components} and Proposition~\ref{p:graded} together,
this is because the primary components are \emph{finely graded}: under
some grading by a free abelian group, namely $\ZZ\Phi$, the vector
space dimensions of the graded pieces of the quotient modulo the
ideal~$\cC$ are all $0$ or~$1$
\cite[Proposition~1.11]{binomialideals}.  In fact, via
Lemma~\ref{l:distinct}, fine gradation is the root cause of
primaryness.
\end{remark}

\begin{remark}
Theorem~\ref{t:components} easily generalizes to arbitrary binomial
ideals in arbitrary commutative noetherian semigroup rings over~$\CC$:
simply choose a presentation as a quotient of a polynomial ring modulo
a pure difference binomial ideal \cite[Theorem~7.11]{gilmer}.
\end{remark}

\begin{remark}
The methods of Section~\ref{s:semigroup} work in arbitrary
characteristic---and indeed, over a field $\kk$ that can fail to be
algebraically closed, and can even be finite---because we assumed that
a prime ideal $\pp_\Phi$ for a face~$\Phi$ is associated to our
binomial ideal.  In contrast, this section and the next work only over
an algebraically closed field of characteristic zero.  However, it
might be possible to produce similarly explicit binomial primary
decompositions in positive characteristic by reducing to the situation
in Section~\ref{s:semigroup}; this remains an open problem.
\end{remark}

\section{Associated components and multigradings}
\label{s:toral}

In this section, we turn our attention to interactions of primary
components with various gradings on~$\CC[\ttt]$.  These played crucial
roles already in the proof of Theorem~\ref{t:components}: taking the
quotient of its statement by the toric ideal~$I_\rho$ put us in the
situation of Proposition~\ref{p:graded} and Theorem~\ref{t:zerocomp},
which provide excellent control over gradings.  The methods here can
be viewed as aids for clarification in examples, as we shall see in
the case of lattice basis ideals (Example~\ref{ex:I(B)}).  However,
this theory was developed with applications in mind \cite{dmm}; see
Section~\ref{s:applications}.

Generally speaking, given a grading of~$\CC[\ttt]$, there are two
kinds of graded modules: those with bounded Hilbert function (the
\emph{toral} case below) and those without.  The main point is
Theorem~\ref{t:toral}: if $\CC[\ttt]/\pp$ has bounded Hilbert function
for some graded prime~$\pp$, then the $\pp$-primary component of any
graded binomial ideal is easier to describe than usual.

To be consistent with notation, we adopt the following conventions for
this section.

\begin{convention}\label{conv:A}
$A = (a_{ij}) \in \ZZ^{d \times n}$ denotes an integer $d \times n$
matrix of rank~$d$ whose columns $a_1,\ldots,a_n$ all lie in a single
open linear half-space of~$\RR^d$; equivalently, the cone generated by
the columns of $A$ is pointed (contains no lines), and all of the
columns~$a_i$ are nonzero.  We also assume that $\ZZ A = \ZZ^d$; that
is, the columns of $A$ span $\ZZ^d$ as a lattice.
\end{convention}

\begin{convention}\label{conv:B}
Let $B = (b_{jk})\in \ZZ^{n\times m}$ be an integer matrix of full
rank~$m \leq n$.  Assume that every nonzero element of the column span
$\ZZ B$ of~$B$ over the integers~$\ZZ$ is \emph{mixed}, meaning that
it has at least one positive and one negative entry; in particular,
the columns of $B$ are mixed.  We write $b_1,\ldots,b_n$ for the rows
of~$B$.  Having chosen~$B$, we set $d = n - m$ and pick a matrix $A
\in \ZZ^{d \times n}$ such that $AB = 0$ and $\ZZ A = \ZZ^d$.  

If $d\neq 0$, the mixedness hypothesis on $B$ is equivalent to
the pointedness assumption for~$A$ in Convention~\ref{conv:A}.  We do
allow~\mbox{$d=0$}, in which case $A$ is the empty matrix.
\end{convention}

The $d \times n$ integer matrix~$A$ in Convention~\ref{conv:A}
determines a $\ZZ^d$-grading on $\CC[\ttt]$ in which the degree
\mbox{$\deg(\del_j) = a_j$} is defined%
	\footnote{In noncommutative settings, such as \cite{MMW,dmm},
	the variables are written $\ddel_1,\ldots,\ddel_n$, and the
	degree of $\ddel_j$ is usually defined to be $-a_j$ instead
	of~$a_j$.}
to be the $j^\th$ column of~$A$.  Our conventions imply that
$\CC[\ttt]$ has finite-dimensional graded pieces, like any finitely
generated module \mbox{\cite[Chapter~8]{cca}}.

\begin{defn}\label{d:toral}
Let $V = \bigoplus_{\alpha \in \ZZ^d} V_\alpha$ be an $A$-graded
module over the polynomial ring~$\CC[\ttt]$.  The \emph{Hilbert
function} $H_V: \ZZ^d \to \NN$ takes the values $H_V(\alpha) =
\dim_\CC{V_\alpha}$.  If $V$ is finitely generated, we say that the module
$V$ is \emph{toral}\/ if the Hilbert function $H_V$ is bounded above.
A graded prime~$\pp$ is a \emph{toral prime} if $\CC[\ttt]/\pp$ is a
toral module.  Similarly, a graded primary component $J$ of an
ideal~$I$ is a \emph{toral component} of~$I$ if $\CC[\ttt]/J$ is a
toral module.
\end{defn}

\begin{example}
The toric ideal $I_A$ for the grading matrix~$A$ is always an
$A$-graded toral prime, since the quotient $\CC[\ttt]/I_A$ is always
toral: its Hilbert function takes only the values $0$ or~$1$.  In
contrast, $\CC[\ttt]$ itself is not a toral module unless $d = n$
(which forces $A$ to be invertible over~$\ZZ$, by
Convention~\ref{conv:A}).
\end{example}

We will be most interested in the quotients of $\CC[\ttt]$ by prime
and primary binomial ideals.  To begin, here is a connection between
the natural gradings from Section~\ref{s:semigroup} and the
$A$-grading.

\begin{lemma}\label{l:toral}
Let $I \subseteq \CC[\ttt]$ be an $A$-graded binomial ideal and
$\cC_{\rho,J}$ a primary component, with $\rho : L \to \CC^*$ for $L
\subseteq \ZZ^J$.  The image $\ZZ A_J$ of the homomorphism $\ZZ^J/L =
\ZZ\Phi \to \ZZ^d$ induced by Corollary~\ref{c:graded} (with $\cA =
\ZZ A = \ZZ^d$) is generated by the columns $a_j$ of~$A$ indexed by~$j
\in J$, as is the monoid image of\/ $\Phi = \NN^J/L$, which we denote
by~$\NN A_J$.\qed
\end{lemma}

To make things a little more concrete, let us give one more
perspective on the homomorphism $\ZZ\Phi \to \ZZ^d$.  Simply put, the
ideal $I_{\rho,J}$ is naturally graded by $\ZZ^J/L = \ZZ\Phi$, and the
fact that it is also $A$-graded means that $L \subseteq \ker(\ZZ^n \to
\ZZ^d)$, the map to $\ZZ^d$ being given by~$A$.  (The real content of
Corollary~\ref{c:graded} lies with the action on the rest of~$\cB$.)

\begin{example}
Let $\rho : L \to \CC^*$ for a saturated sublattice $L \subseteq \ZZ^J
\subseteq \ZZ^n$.  If $\cC_{\rho,J}$ is an $I_{\rho,J}$-primary
binomial ideal, then $\CC[\ttt]/\cC_{\rho,J}$ has a finite filtration
whose successive quotients are torsion-free modules of rank~$1$ over
the affine semigroup ring $R = \CC[\ttt]/I_{\rho,J}$.  This follows by
applying Proposition~\ref{p:graded} to Theorem~\ref{t:components}.1
and its proof.  If, in addition, $I_{\rho,J}$ is $A$-graded, then some
$A$-graded translate of each successive quotient admits a
$\ZZ^J/L$-grading refining the $A$-grading via $\ZZ^J/L \to \ZZ^d =
\ZZ A$; this follows by conjointly applying Corollary~\ref{c:graded}.
\end{example}

The next three results provide alternate characterizations of toral
primary binomial ideals.  In what follows, $A_J$ is the submatrix
of~$A$ on the columns indexed~by~$J$.

\begin{prop}\label{p:toral}
Every $A$-graded toral prime is binomial.  In the situation of
Lemma~\ref{l:toral}, $\CC[\ttt]/I_{\rho,J}$ and
$\CC[\ttt]/\cC_{\rho,J}$ are toral if and only if the homomorphism
$\ZZ\Phi \to \ZZ^d$ is injective.
\end{prop}

\begin{proof}
To prove the first part of the statement, fix a toral prime~$\pp$, and
let $h \in \NN$ be the maximum of the Hilbert function
of~$\CC[\ttt]/\pp$.  It is enough, by
\cite[Proposition~1.11]{binomialideals}, to show that $h = 1$.
Let~$R$ be the localization of~$\CC[\ttt]/\pp$ by inverting all
nonzero homogeneous elements.  Because of the homogeneous units
in~$R$, all of its graded pieces have the same dimension over~$\CC$;
and since $R$ is a domain, this dimension is at least~$h$.  Thus we
need only show that $R_0 = \CC$.  For any given finite-dimensional
subspace of~$R_0$, multiplication by a common denominator maps it
injectively to some graded piece of $\CC[\ttt]/\pp$.  Therefore every
finite-dimensional subspace of~$R_0$ has dimension at most~$h$.  It
follows that $H_R(0) \leq h$, so $R_0$ is artinian.  But $R_0$ is a
domain because $R_0 \subseteq R$, so $R_0 = \CC$.

For the second part, $\CC[\ttt]/\cC_{\rho,J}$ has a finite filtration
whose associated graded pieces are $A$-graded translates of quotients
of $\CC[\ttt]$ by $A$-graded primes, at least one of which is
$I_{\rho,J}$ and all of which contain~it.  By additivity of Hilbert
functions, $\CC[\ttt]/\cC_{\rho,J}$ is toral precisely when all of
these are toral primes.  However, if a graded prime~$\pp$ contains a
toral prime, then $\pp$ is itself a toral prime.  Therefore, we need
only treat the case of $\CC[\ttt]/I_{\rho,J}$.  But
$\CC[\ttt]/I_{\rho,J}$ is naturally graded by~$\ZZ\Phi$, with Hilbert
function $0$ or~$1$, so injectivity immediately implies that
$\CC[\ttt]/I_{\rho,J}$ is toral.  On the other hand, if $\ZZ\Phi \to
\ZZ^d$ is not injective, then $\NN A_J$ is a proper quotient of the
affine semigroup~$\Phi$, and such a proper quotient has fibers of
arbitrary cardinality.
\end{proof}

\begin{cor}\label{c:toralprimes}
Let $\rho : L \to \CC^*$ for a saturated lattice $L \subseteq \ZZ^J
\cap \ker_{\ZZ}(A) = \ker_{\ZZ}(A_J)$.  The quotient\/
$\CC[\ttt]/I_{\rho,J}$ by an $A$-graded prime~$I_{\rho,J}$
is toral if and only if $L = \ker_{\ZZ}(A_J)$.\qed
\end{cor}

\begin{lemma}\label{l:toral-by-A}
Every $A$-graded binomial prime ideal $I_{\rho,J}$ satisfies
\[
\dim(I_{\rho,J}) \geq \rank(A_J),
\]
with equality if and only if\/ $\CC[\ttt]/I_{\rho,J}$ is toral.
\end{lemma}
\begin{proof}
Rescale the variables and assume that $I_{\rho,J} = I_L$, the lattice
ideal for a saturated lattice $L \subseteq \ker_{\ZZ}(A_J)$.  The rank
of $L$ is at most $\# J - \rank(A_J)$; thus $\dim(I_L) = \# J -
\rank(L) \geq \rank(A_J)$.  Equality holds exactly when
$L=\ker_{\ZZ}(A)$, i.e.\ when $\CC[\ttt]/I_{\rho,J}$ is toral.
\end{proof}

\begin{example}\label{ex:I(B)}
Fix matrices $A$ and~$B$ as in Convention~\ref{conv:B}.  This
identifies $\ZZ^d$ with the quotient of\/ $\ZZ^n/\ZZ B$ modulo its
torsion subgroup.  Consider the \emph{lattice basis ideal}\/
\begin{equation}\label{eq:I(B)}
  I(B) = \<\ttt^{u_+} - \ttt^{u_-} \mid u = u_+ - u_- \;\mbox{is a
  column of}\; B\> \subseteq \CC[\del_1,\dots ,\del_n].
\end{equation}
The toric ideal~$I_A$ from~(\ref{eq:IA}) is an associated prime
of~$I(B)$, the primary component being $I_A$ itself.  More generally,
all of the minimal primes of the lattice ideal~$I_{\ZZ B}$, one of
which is~$I_A$, are minimal over~$I(B)$ with multiplicity~$1$; this
follows from \cite[Theorem~2.1]{binomialideals} by inverting the
variables.  That result also implies that the minimal primes
of~$I_{\ZZ B}$ are precisely the ideals~$I_\rho$ for partial
characters $\rho: \sat(\ZZ B) \to \CC^*$ of~$\ZZ^n$ extending the
trivial partial character on~$\ZZ B$, so the lattice ideal~$I_{\ZZ B}$
is the intersection of these prime ideals.  Hence $I_{\ZZ B}$ is a
radical ideal, and every irreducible component of its zero set is
isomorphic, as a subvariety of $\CC^n$, to the variety of~$I_A$.

In complete generality, each of the minimal primes of~$I(B)$ arises,
after row and column permutations, from a block decomposition of~$B$
of the form
\begin{equation}\label{eq:MNOB}
  \left[
    \begin{array}{l|r}
    N & B_J\!\\\hline
    M & 0\ 
    \end{array}
  \right],
\end{equation}
where $M$ is a mixed submatrix of~$B$ of size $q \times p$ for some $0
\leq q \leq p \leq m$ \cite{hostenshapiro}.  (Matrices with $q = 0$
rows are automatically mixed; matrices with $q = 1$ row are never
mixed.)  We note that not all such decompositions correspond to 
minimal primes: the matrix $M$ has to satisfy another condition which
Ho\c{s}ten and Shapiro call irreducibility \cite[Definition~2.2 and
Theorem~2.5]{hostenshapiro}.  If $I(B)$ is a complete intersection,
then only square matrices $M$ will appear in the block
decompositions~(\ref{eq:MNOB}), by a result of Fischer and Shapiro
\cite{fischer-shapiro}.

For each partial character $\rho : \sat(\ZZ B_J) \to \CC^*$ extending
the trivial character on~$\ZZ B_J$, the prime $I_{\rho,J}$ is
associated to~$I(B)$, where $J = J(M) = \{1,\ldots,n\} \minus {\rm
rows}(M)$ indexes the $n-q$ rows not in~$M$.  We reiterate that the
symbol $\rho$ here includes the specification of the sublattice
$\sat(\ZZ B_J) \subseteq \ZZ^n$.  The corresponding primary component
\mbox{$\cC_{\rho,J} = {\rm Hull}\big(I(B) + I_\rho + \mm_J^e\big)$} of
the lattice basis ideal~$I(B)$ is simply $I_\rho$ if $q = 0$, but will
in general be non-radical when $q \geq 2$ (recall that $q = 1$ is
impossible).

The quotient $\CC[\ttt]/\cC_{\rho,J}$ is toral if and only if $M$ is
square and satisfies either $\det(M) \neq 0$ or $q=0$.  To check this
statement, observe that $I(B)$ has $m = n - d$ generators, so the
dimension of any of its associated primes is at least~$d$.  But since
$A_J$ has rank at most~$d$, Lemma~\ref{l:toral-by-A} implies that
toral primes of~$I(B)$ have dimension exactly~$d$ (and are therefore
minimal).  If $I_{\rho,J}$ is a toral associated prime of $I(B)$
arising from a decomposition of the form~(\ref{eq:MNOB}), where $M$ is
$q \times p$, then the dimension of $I_{\rho,J}$ is $n-p-(m-q) =
d+q-p$, and from this we conclude that $M$ is square.  That $M$ is
invertible follows from the fact that $\rank(\ker_{\ZZ}(A_J)) = d$.
The same arguments show that if $M$ is not square invertible, then
$I_{\rho,J}$ is not toral.
\end{example}

\begin{example}
A binomial ideal $I \subseteq \CC[\ttt]$ may be $A$-graded for
different matrices $A$; in this case, which of the components of $I$
are toral will change if we alter the grading.  For instance, the
prime ideal $I = \<\del_1\del_4-\del_2\del_3\> \subseteq
\CC[\del_1,\dots,\del_4]$ is homogeneous for both the matrix
$[1 \, 1 \, 1 \, 1]$ and the matrix
$\left[\begin{smallmatrix}
  1 & 1 & 1 & 1 \\
  0 & 1 & 0 & 1 \\
  0 & 0 & 1 & 1
\end{smallmatrix}\right]$.
But $\CC[\del_1,\dots,\del_4]/I$ is toral in the
$\left[\begin{smallmatrix}
  1 & 1 & 1 & 1 \\
  0 & 1 & 0 & 1 \\
  0 & 0 & 1 & 1
\end{smallmatrix}\right]$-grading,
while it is not toral in the $[1 \, 1 \, 1 \, 1]$-grading.
\end{example}

\begin{example}\label{ex:binomial}
Let $I = \<bd-de,bc-ce,ab-ae,c^3-ad^2,a^2d^2-de^3,a^2cd-ce^3,a^3d-ae^3\>$
be a binomial ideal in $\CC[\ttt]$, where we write $\ttt =
(\del_1,\del_2,\del_3,\del_4,\del_5) = (a,b,c,d,e)$, and let
\[
  A = \left[\begin{array}{ccccc}
	1 & 1 & 1 & 1 & 1 \\
	0 & 1 & 2 & 3 & 1
      \end{array}\right]
\qquad\text{and}\qquad
  B = \left[\begin{array}{rrr}
	-2 & -1 &  0 \\ 
	 3 &  0 &  1 \\
	 0 &  3 &  0 \\ 
	-1 & -2 &  0 \\
	 0 &  0 & -1
      \end{array}\right].
\]
One easily verifies that the binomial ideal~$I$ is graded by $\ZZ A =
\ZZ^2$.  If $\omega$ is a primitive cube root of unity ($\omega^3 =
1$), then $I$, which is a radical ideal, has the prime decomposition
\begin{align*}
I =
  \<a,c,d\> &\cap\<bc-ad,b^2-ac,c^2-bd,b-e\>
\\          &\cap\<\omega bc-ad,b^2-\omega ac,\omega^2 c^2-bd,b-e\>
\\          &\cap\<\omega^2bc-ad,b^2-\omega^2ac,\omega c^2-bd,b-e\>.
\end{align*}
The intersectand $\<a,c,d\>$ equals the prime ideal $I_{\rho,J}$ for
$J = \{2,5\}$ and $L = \{0\} \subseteq \ZZ^J$.  The homomorphism
$\ZZ^J \to \ZZ^2$ is not injective since it maps both basis vectors to
$\left[\twoline{1}{1}\right]$; therefore the prime ideal $\<a,c,d\>$
is not a toral component of~$I$.  In contrast, the remaining three
intersectands are the prime ideals $I_{\rho,J}$ for the three
characters $\rho$ that are defined on~$\ker(A)$ but trivial on its
index~$3$ sublattice~$\ZZ B$ spanned by the columns~of $B$, where $J =
\{1,2,3,4,5\}$.  These prime ideals are all toral by
Corollary~\ref{c:toralprimes}, with $\ZZ A_J = \ZZ A$.
\end{example}

Toral components can be described more simply than in
Theorem~\ref{t:components}.

\begin{theorem}\label{t:toral}
Fix an $A$-graded binomial ideal $I \subseteq \CC[\ttt]$ and a toral
associated prime $I_{\rho,J}$ of~$I$.  Define the binomial ideal $\oI
= I \cdot \CC[\ttt]/\<\del_j - 1 \mid j \in J\>$ by setting $\del_j =
1$ for $j \in J$.
\begin{numbered}
\item
Fix a minimal prime $I_{\rho,J}$ of\/~$I$.  If\/ $\oU \subseteq
\NN^\oJ$ is the set of elements with infinite congruence class
in\/~$\NN^\oJ_\oI$ (Proposition~\ref{p:sim}), and $\ttt_J = \prod_{j
\in J} \del_j$, then $I$ has $I_{\rho,J}$-primary component
\[
  \cC_{\rho,J} = \big( (I + I_\rho) : \ttt_J^\infty\big) + \<\ttt^u
  \mid u \in \oU\>.
\]
\end{numbered}\setcounter{separated}{\value{enumi}}%
Let $K \subseteq \CC[\del_j \mid j \in \oJ]$ be a monomial ideal
containing a power of each available variable, and let $\oU_{\!K}
\subseteq \NN^\oJ$ be the set of elements with infinite congruence
class in\/~$\NN^\oJ_{\oI + K}$.
\begin{numbered}\setcounter{enumi}{\value{separated}}
\item
The $I_{\rho,J}$-primary component of $\<I + I_\rho + K\> \subseteq
\CC[\ttt]$ is $\big( (I + I_\rho) : \ttt_J^\infty \big) + \<\ttt^u
\mid u\in \oU_{\!K}\>$.
\item
If $K$ is contained in a sufficiently high power of\/ $\mm_J$, then
\[
  \cC_{\rho,J} = \big( (I + I_\rho) : \ttt_J^\infty\big) + \<\ttt^u
  \mid u \in \oU_{\!K}\>
\]
is a valid choice of $I_{\rho,J}$-primary component for~$I$.
\end{numbered}
The only monomials in the above primary components are in $\<\ttt^u
\mid u \in \oU\>$ or $\<\ttt^u \mid u \in \oU_{\!K}\>$.
\end{theorem}
\begin{proof}
Resume the notation from the statement and proof of
Theorem~\ref{t:components}.  As in that proof, it suffices here to
deal with the first item.  In fact, the only thing to show is that
$\wU$ in Theorem~\ref{t:components} is the same as $\NN^J \times \oU$
here.

Recall that $I' \subseteq \CC[Q]$ is the image of~$I$ modulo~$I_\rho$.
The congruence classes of $\ZZ\Phi \times \NN^\oJ$ determined by
$I'[\ZZ\Phi]$ are the projections under $\ZZ^J \times \NN^\oJ \to
\ZZ\Phi \times \NN^\oJ$ of the $\sim$ congruence classes.  Further
projection of these classes to~$\NN^\oJ$ yields the congruence classes
determined by the ideal $I'' \subseteq \CC[\NN^\oJ]$, where $I''$ is
obtained from $I'[\ZZ\Phi]$ by setting $\ttt^\phi = 1$ for all
\mbox{$\phi \in \ZZ\Phi$}.  This ideal $I''$ is just~$\oI$.  Hence
we are reduced to showing that a congruence class in $\Phi \times
\NN^\oJ$ determined by $I'[\ZZ\Phi]$ is infinite if and only if its
projection to~$\NN^\oJ$ is infinite.  This is clearly true for the
monomial congruence class in $\ZZ\Phi \times \NN^\oJ$.  For any other
congruence class $\Gamma \subseteq \ZZ\Phi \times \NN^\oJ$, the
homogeneity of~$I$ (and hence that of~$I'$) under the $A$-grading
implies that $\Gamma$ is contained within a coset of $\KK =
\ker(\ZZ\Phi \times \ZZ^\oJ \to \ZZ^d=\ZZ A)$.  This kernel~$\KK$
intersects $\ZZ\Phi$ only at $0$ because $I_{\rho,J}$ is toral.
Therefore the projection of any coset of~$\KK$ to~$\ZZ^\oJ$ is
bijective onto its image.  In particular, $\Gamma$ is infinite if and
only if its bijective image in $\NN^\oJ$~is~infinite.%
\end{proof}

\begin{cor}\label{c:I(B)}
Resume the notation of Example~\ref{ex:I(B)}.  If $I_{\rho,J}$ is a
toral minimal prime of the lattice basis ideal~$I(B)$ given by a
decomposition as in~(\ref{eq:MNOB}), so $J = J(M)$, then
\[
  \cC_{\rho,J} = I(B) + I_{\rho,J} + U_M,
\]
where $U_M \subseteq \CC[\del_j \mid j \in \oJ]$ is the ideal
$\CC$-linearly spanned by by all monomials whose exponent vectors lie
in the union of the unbounded $M$-subgraphs of\/~$\NN^\oJ$, as in
Definition~\ref{d:M}.  The only monomials in $\cC_{\rho,J}$ belong to
$U_M$.\qed
\end{cor}

\begin{remark}
Theorem~\ref{t:toral} need not always be false for a component that is
not toral, but it can certainly fail: there can be congruence classes
in $\ZZ\Phi \times \NN^\oJ$ that are infinite only in the $\ZZ\Phi$
direction, so that their projections to~$\NN^\oJ$ are finite.
\end{remark}

\section{Applications, examples, and further directions}
\label{s:applications}

In this section we give a brief overview of the connection between 
binomial primary decomposition and hypergeometric differential 
equations, study some examples, and discuss computational issues.

From the point of view of complexity, primary decomposition is hard:
even in the case of zero dimensional binomial complete intersections,
counting the number of associated primes (with or without
multiplicity) is a $\# P$-complete problem
\cite{cattani-dick:binomial-ci}.  However, the primary decomposition
algorithms implemented in Singular \cite{Singular} or Macaulay2
\cite{M2} work very well in reasonably sized examples, and in fact,
they provide the only implemented method for computing bounded
congruence classes or $M$-subgraphs as in Section~\ref{s:primary}.
We remind the reader that \cite[Section 8]{binomialideals} contains
specialized algorithms for binomial primary decomposition, whose 
main feature is that they preserve binomiality at each step. 

In the case that $q=2$, we can study $M$-subgraphs directly by
combinatorial means \cite[Section~6]{dms}.  The relevant result is the
following.

\begin{prop}
Let $M$ be a mixed invertible $2 \times 2$ integer matrix.
Without loss of generality, write
$M =  \left[ \begin{array}{rr} a & b \\ -c & -d \end{array} \right] $,
where $a,b,c,d$ are positive integers.  Then the number of
bounded $M$-subgraphs is $\min(ad, bc)$.  Moreover, if
\[
  R = \left\{
  \begin{array}{lr} 
  \{(s,t) \in \NN^2 \mid s < b \text{ and } t < c  \} & \mbox{if}\; ad > bc, \\
  \{(s,t) \in \NN^2 \mid s < a \text{ and } t < d  \} & \mbox{if}\; ad < bc,
  \end{array}
  \right.
\]
then every bounded $M$-subgraph passes through exactly one of the
points in $R$.
\end{prop}

If $q>2$, a method for computing $M$-subgraphs may be obtained through
a link to differential equations.  To make this evident, we make a
change in the notation for the ambient ring.

\begin{notation}
All binomial ideals in the remainder of this article are ideals in the
polynomial ring $\CC[\ddel]=\CC[\ddel_1,\dots,\ddel_n]$.
\end{notation}

The following result for $M$-subgraphs can be adapted to fit the more
general context of congruences.

\begin{prop}\label{prop:sols-via-subgraphs}
Let $M$ be a $q\times q$ mixed invertible integer matrix, and assume
that $q>0$.  Given $\gamma \in \NN^q$, denote by $\Gamma$ the
$M$-subgraph containing $\gamma$.  Think of the ideal $I(M) \subseteq
\CC[\ddel]$ as a system of linear partial differential equations with
constant coefficients.
\begin{numbered}
\item
The system of differential equations $I(M)$ has a unique formal power
series solution of the form $G_{\gamma} = \sum_{u \in \Gamma}
\lambda_u x^u$ in which $\lambda_{\gamma} = 1$.
\item
The other coefficients $\lambda_u$ of $G_{\gamma}$ for $u \in \Gamma$
are all nonzero.
\item
The set $\{G_{\gamma} \mid \gamma$ runs over a set of representatives
for the $M$-subgraphs of\/~$\NN^q\}$ is a basis for the space of all
formal power series solutions of~$I(M)$.
\item
The set $\{G_{\gamma} \mid \gamma$ runs over a set of representatives
for the \emph{bounded} $M$-subgraphs of~$\NN^q\}$ is a basis for the
space of polynomial solutions of~$I(M)$.
\end{numbered}
\end{prop}

The straightforward proof of this proposition can be found in
\cite[Section~7]{dmm}.

The following example illustrates the correspondence between
$M$-subgraphs and solutions of $I(M)$.

\begin{example}
\label{ex:concrete-subgraph}
Consider the $3\times 3$ matrix
\[
M =
  \left[
  \begin{array}{rrr}
    1 & -5 &  0 \\
   -1 &  1 & -1 \\
    0 &  3 &  1
  \end{array}
  \right].
\]
A basis of solutions (with minimal support under inclusion) of $I(M)$
is easily computed:
\[
\left\{
  1, \quad x+y+z, \quad  (x+y+z)^2, \quad (x+y+z)^3, \quad
  \sum_{n \geq 4} \frac{(x+y+z)^n}{n!}
\right\}.
\]
\begin{figure} 
\[
\psfrag{a}{\footnotesize$a$}
\psfrag{b}{\footnotesize$b$}
\psfrag{c}{\footnotesize$c$}
\includegraphics{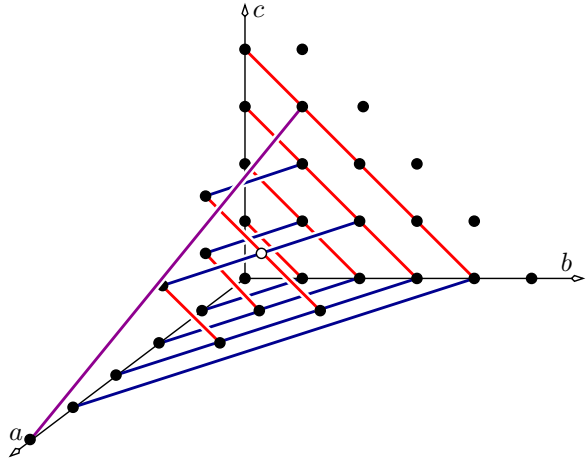}
\]
\caption{The $M$-subgraphs of $\NN^3$ for Example~\ref{ex:concrete-subgraph}.}
\label{f:M-subgraphs}
\end{figure}

The $M$-subgraphs of $\NN^3$ are the four slices $\{(a,b,c) \in \NN^3
\mid a + b + c = n\}$ for $n \leq 3$; for $n \geq 4$, two consecutive
slices are $M$-connected by  $(-5,1,3)$,  yielding one unbounded
$M$-subgraph (see Figure~\ref{f:M-subgraphs}).
\end{example}

A direct combinatorial algorithm for producing the bounded
$M$-subgraphs for $q > 2$, or even for finding their number would be
interesting and useful, as the number of bounded $M$-subgraphs gives
the dimension of the polynomial solution space of a hypergeometric
system, and also the multiplicity of an associated prime of a lattice
basis ideal.  In the case where $I(M)$ is a zero-dimensional complete
intersection, such an algorithm can be produced from the results in
\cite{cattani-dick:binomial-ci}.  The combinatorial computation of the
number of bounded congruence classes determined by a binomial ideal in
a semigroup ring is open.

The system of differential equations $I(M) \subseteq\CC[\ddel]$ is a
special case in the class of \emph{Horn hypergeometric systems}.  That
class of systems takes center stage in our companion article
\cite{dmm}, in the more general setting of \emph{binomial $D$-modules} that
are introduced there.  
The input data for these consist of a binomial ideal~$I$ and a
vector~$\beta$ of complex parameters.  The special case where $I$ is
prime corresponds to the \emph{$A$-hypergeometric} or \emph{GKZ
hypergeometric systems}, after Gelfand, Kapranov, and Zelevinsky
\cite{GGZ, GKZ}.  Binomial primary decomposition is crucial for the
study of Horn systems, and their more general binomial relatives,
because the numerics, algebra, and combinatorics of their solutions
are directly governed by the corresponding features of the input
binomial ideal.  The dichotomy between components that are toral or
not, for example, distinguishes between the choices of parameters
yielding finite- or infinite-dimensional solution spaces; and in the
finite case, the multiplicities of the toral components enter into
simple formulas for the dimension.  The use of binomial primary
decomposition to extract invariants and reduce to the
$A$-hypergeometric case underlies the entirety of~\cite{dmm}; see
Section~1.6 there for precise statements.

Using this ``hypergeometric perspective'', it
becomes possible to consider algorithms for binomial primary
decomposition that exploit methods for solving differential equations.
The idea would be to use the fact from \cite[Section 7]{dmm}, that the
supports of power series solutions of hypergeometric systems contain
the combinatorial information needed to describe certain toral primary
components of the underlying binomial ideal.  The method of canonical
series \cite[Sections~2.5 and~2.6]{SST}, a symbolic algorithm for
constructing power series solutions of regular holonomic left
$D_n$-ideals, might be useful.  Canonical series methods are based on
Gr\"obner bases in the Weyl algebra, and generalize the classical
Frobenius method in one variable.


\raggedbottom
\def\cprime{$'$} \def\cprime{$'$}
\providecommand{\MR}{\relax\ifhmode\unskip\space\fi MR }
\providecommand{\MRhref}[2]{%
  \href{http://www.ams.org/mathscinet-getitem?mr=#1}{#2}
}
\providecommand{\href}[2]{#2}

\end{document}